\documentclass[12pt,a4paper,psamsfonts]{amsart}
\usepackage{amssymb,amscd,amsxtra,calc}
\usepackage{cmmib57}

\newcommand{\widebar}[1]{\hspace{0.5mm}\overline{\hspace{-0.5mm}#1\hspace{-0.5mm}}\hspace{0.5mm}}

\theoremstyle{plain}
    \newtheorem{thm}{Theorem}[section]
    
    \newtheorem{claim}[thm]{Claim}
    \newtheorem{corollary}[thm]{Corollary}
    
    \newtheorem{lemma}[thm]{Lemma}

    \newtheorem{theorem}[thm]{Theorem}

\theoremstyle{definition}

    \newtheorem{remark}[thm]{Remark}
\theoremstyle{remark}

\theoremstyle{question}

    \newtheorem{setup}[thm]{}

\newcommand{\C}{\mathbb{C}}

\newcommand{\BPP}{\mathbb{P}}

\newcommand{\Q}{\mathbb{Q}}

\newcommand{\R}{\mathbb{R}}

\newcommand{\Z}{\mathbb{Z}}

\newcommand{\OO}{\mathcal{O}}

\newcommand{\Aut}{\operatorname{Aut}}

\newcommand{\block}{\operatorname{block}}
\newcommand{\diag}{\operatorname{diag}}

\newcommand{\Gal}{\operatorname{Gal}}
\newcommand{\GL}{\operatorname{GL}}
\newcommand{\group}{\operatorname{group}}

\newcommand{\id}{\operatorname{id}}
\newcommand{\Imm}{\operatorname{Im}}
\newcommand{\Ker}{\operatorname{Ker}}
\newcommand{\NE}{\overline{\operatorname{NE}}}
\newcommand{\Nef}{\operatorname{Nef}}
\newcommand{\NS}{\operatorname{NS}}

\newcommand{\Proj}{\operatorname{Proj}}

\newcommand{\rank}{\operatorname{rank}}
\newcommand{\rk}{\operatorname{rk}}
\newcommand{\Sing}{\operatorname{Sing}}
\newcommand{\SL}{\operatorname{SL}}
\newcommand{\SSO}{\operatorname{SO}}
\newcommand{\Stab}{\operatorname{Stab}}

\newcommand{\torsion}{\operatorname{torsion}}

\newcommand{\variety}{\operatorname{variety}}

\newcommand{\Alb}{\operatorname{Alb}}

\begin{document}

\title[Automorphism groups of positive entropy on projective threefolds]{
Automorphism groups of positive entropy on projective threefolds}

\author{Frederic Campana}
\address
{
\textsc{Department of Mathematics} \endgraf
\textsc{University of Nancy 1,
BP 239,
F-54506, Vandoeuvre-les-Nancy, Cedex,
France}}
\email{Frederic.Campana@iecn.u-nancy.fr}

\author{Fei Wang}
\address
{
\textsc{Department of Mathematics} \endgraf
\textsc{National University of Singapore,
10 Lower Kent Ridge Road,
Singapore 119076
}}
\email{matwf@nus.edu.sg}

\author{De-Qi Zhang}
\address
{
\textsc{Department of Mathematics} \endgraf
\textsc{National University of Singapore,
10 Lower Kent Ridge Road,
Singapore 119076
}}
\email{matzdq@nus.edu.sg}

\begin{abstract}
We prove two results about the natural representation of a group $G$
of automorphisms of a normal projective threefold $X$
on its second cohomology. We show that if $X$ is minimal then $G$, modulo a normal subgroup
of null entropy, is embedded as a Zariski-dense subset in a semi-simple real linear algebraic group of
real rank $\le 2$.
Next, we show that $X$ is a complex torus if the image of $G$ is an almost abelian group of positive rank
and the kernel is infinite,
unless $X$ is equivariantly non-trivially fibred.
\end{abstract}

\subjclass[2000]{
32H50, 
14J50, 
32M05, 
37B40 
}
\keywords{automorphism, iteration, complex dynamics, topological entropy}


\maketitle

\section{Introduction}\label{Intro}

Let $X$ be a compact K\"ahler manifold.
For an automorphism $g \in \Aut(X)$, its (topological)
{\it entropy} $h(g) = \log \rho(g)$ is defined as
the logarithm of the {\it spectral radius} $\rho(g)$
of the pullback action $g^*$ on the total cohomology group of $X$, i.e.,
$$\rho(g) := \max \{|\lambda| \, ; \, \lambda \,\,\,
\text{is an eigenvalue of} \,\,\,
g^* | \oplus_{i \ge 0} H^i(X, \C)\}.$$
By the fundamental result of Gromov and Yomdin,
the above definition is equivalent to the original dynamical definition of entropy
(cf.\ \cite{Gr}, \cite{Yo}).

An element $g \in \Aut(X)$ is of {\it null entropy}
if its (topological) entropy $h(g)$ equals $0$.
For a subgroup $G$ of $\Aut(X)$, we define the {\it null subset} of $G$ as
$$N(G) := \{g \in G \, | \, g \, \,
\text{\rm is of null entropy, i.e.,} \, h(g) = 0\} $$
which may {\it not} be a subgroup.
A group $G \le \Aut(X)$ is of {\it null entropy} if every $g \in G$
is of null entropy, i.e., if $G$ equals $N(G)$.

By the classification of surfaces, a complex surface $S$ has some
$g \in \Aut(S)$ of positive entropy only if
$S$ is bimeromorphic to a rational surface, complex torus,
$K3$ surface or Enriques surface (cf. \cite{Ca99}).
See \cite{JDG} for a similar phenomenon in higher dimensions.

Recall that a normal projective variety $X$ is {\it minimal} if it has at worst
terminal singularities and the canonical divisor $K_X$ is nef
(cf.\ \cite[Definition 2.34]{KM}).
Let $\NS(X)$ be the {\it Neron-Severi group}
and $\NS_{\C}(X) := \NS(X) \otimes \C$.
For a subgroup $G$ of $\Aut(X)$,
let $\overline{G} \subseteq \GL(\NS_{\C}(X))$ be the Zariski-closure
of the action of $G$ on $\NS_{\C}(X)$ (simply denoted by $G \, | \, \NS_{\C}(X)$), and let
$R(\overline{G})$ be its solvable radical,
both of which are defined over $\Q$ (cf.~\cite[ChI; 0.11, 0.23]{Ma}).
We have
a natural composition of homomorphisms:
$\iota: G \to G \, | \, \NS_{\C}(X) \to \overline{G}$.
Denote by
$$R(G) := \iota^{-1}(\iota(G) \, \cap \, R(\widebar{G})) \, \lhd \, G .$$

\begin{theorem}\label{Cor3}
Let $X$ be a $3$-dimensional minimal projective variety and
$G \le \Aut(X)$ a subgroup
such that $G \, | \, \NS_{\C}(X)$ is not virtually solvable.
Then $R(G) \, | \, \NS_{\C}(X)$ is virtually unipotent.
Replacing $G$ by a suitable finite-index subgroup,
$G/R(G)$ is embedded as a Zariski-dense subgroup in
$H := \overline{G}/R(\overline{G})$ so that $H(\R)$ is
a semi-simple real linear algebraic group
and is either of real rank $1$ {\rm (cf. \cite[0.25]{Ma})}
or locally isomorphic to $\SL_3(\R)$ or $\SL_3(\C)$
$($where {\it locally isomorphic}
means: having isomorphic  Lie algebras$)$.
\end{theorem}

The key step of the proof is Theorem \ref{Ssimple} of which part (1) is a
consequence of \cite[Theorem 5.1]{CZ}
in which the authors have determined the actions of irreducible lattices
in semi-simple real Lie groups of higher rank on threefolds.

\par \vskip 0.5pc
For a subgroup $G$ of $\Aut(X)$, the pair $(X, G)$ is {\it non-strongly-primitive},
if there are $X'$ bimeromorphic to $X$, a finite-index subgroup $G_1$ of $G$
and a holomorphic map $X' \to Y$ with $0 < \dim Y < \dim X$, such that
the induced bimeromorphic action of $G_1$ on $X'$ is biholomorphic and
descends to an action on $Y$ with $X' \to Y$ being $G_1$-equivariant.
$(X, G)$ is {\it strongly primitive} if it is not non-strongly-primitive.

\par \vskip 0.5pc
Our second main result is Theorem \ref{ThF} (being generalized to higher dimensions in \cite{FZ}).

\begin{theorem}\label{ThF}
Let $X$ be a $3$-dimensional normal projective variety with only
$\Q$-factorial terminal singularities, and $G \le \Aut(X)$ a subgroup
such that $G_0 := G \cap \Aut_0(X)$ is infinite
and the quotient group $G/G_0$ is an almost abelian group of positive rank
$($cf. $\ref{setup2.1}$ for the terminology$)$.
Suppose that the pair $(X, G)$ is strongly primitive.
Then $X$ is a complex $3$-torus and $G_0$ is Zariski-dense
in $\Aut_0(X)$.
\end{theorem}

We remark that the almost abelian condition (as defined in $\ref{setup2.1}$)
on $G/G_0$ is used to show that
the extremal rays on $X$ are $G$-periodic
(cf.~\cite[Theorem 2.13, or Appendix]{uniruled}).

\begin{corollary}\label{Cor2}
Let $X$ be a $3$-dimensional normal projective variety with only
$\Q$-factorial
rational singularities,
and $G \le \Aut(X)$ a subgroup of {\rm null entropy}
such that $G \, | \, \NS_{\C}(X)$ is almost abelian of positive rank.
Assume that
$(X, G)$ is strongly primitive. Then, $\Aut_0(X) = \{\id_X\}$ and $h^1(X, \OO_X) = 0$.
\end{corollary}

We do not have any example satisfying all the hypotheses of Corollary \ref{Cor2}.

Let $\tau$ be a primitive cubic or quartic root of $1$, $E := \C/(\Z + \Z \tau)$ and
$X := E^n /\langle \diag[\tau, \dots, \tau]\rangle$
(cf. \cite[Thm (0.3)]{OS}, \cite[Ex 1.7]{CY3}).
Take some $g$ in $\SL_n(\Z)$ such that $g$ acts on $X$ as an automorphism
of infinite order and null entropy.
When $n = 3$,
the group $G := \langle g \rangle$
satisfies the hypotheses of Corollary \ref{Cor2},
except the
strong primitivity which seems hard to verify.
The action of $\langle g \rangle$ on $E^n$
is not strongly primitive (cf. Proof of Cor. \ref{Cor1}).

\begin{remark}\label{Rem1}
(1) In Theorem \ref{Cor3}, suppose that $c_1(X) \ne 0$. Then
the Iitaka fibration $X \to Y$ is $G$-equivariant, and also non-trivial
by the abundance theorem
or the classification of surfaces.
Replacing $G$ by a subgroup of finite index, we may assume that
the induced action of $G$ on $Y$ is trivial (cf.~\cite[Theorem 14.10]{Ue}), so
$G$ acts faithfully on a general fibre $S$ and the group $G \, | \, S$
is not of null entropy (since the same holds for $G \, | \, X$; see Theorem \ref{ThB}).
Since $K_S = K_X | S \sim_{\Q} 0$,
our $S$ is a complex $2$-torus, $K3$ or Enriques surface.
So there is a homomorphism
$G  \, | \, X = G  \, | \, S \to \SSO(1, \rho(S) - 1) \le \SL(\NS_{\R}(S))$
($G$ being replaced by a subgroup of index $\le 2$) with kernel virtually contained in
$\Aut_0(S)$ and the Picard number $\rho(S) \le 20$.

(2) In Theorem \ref{ThF}, the strong primitivity assumption on $(X, G)$ is necessary
by considering $X = S \times T$ and $G = \langle g \rangle \times \Aut_0(T)$
where $g$ is of positive entropy on a $K3$ surface $S$,
and $T$ a homogeneous curve ($\BPP^1$ or elliptic).

(3) The projectivity of $X$ in Theorem \ref{Cor3} is used in applying the characterization
of a quotient of an abelian threefold (cf. \cite{SW}).
The projectivity of $X$ in Theorem \ref{ThF} is used in running the minimal model program
(only for uniruled varieties).
\end{remark}

The following Theorem \ref{ThA} is a direct consequence of \cite[Theorem 1.1]{Z-Tits}; see also
the discussion in \cite[\S 6]{Ca}.
It extends the classical Tits alternative \cite[Theorem 1]{Ti}.

A compact K\"ahler manifold $X$ is {\it ruled}
if it is bimeromorphic to a manifold with
a $\BPP^1$-fibration.
By a result of Matsumura, $X$ is ruled if $\Aut_0(X)$ is not a compact torus
(cf. \cite[Proposition 5.10]{Fu}).
When $X$ is a compact complex K\"ahler manifold (or a normal projective variety),
set $L := H^2(X, \Z)/(\torsion)$ $($resp.\ $L := \NS(X)/(\torsion))$,
$L_{\R} := L \otimes_{\Z} \R$, and $L_{\C} := L \otimes_{\Z} \C$.

\begin{theorem}\label{ThA}
Let $X$ be a compact K\"ahler $($resp.\ projective$)$ manifold
of dimension $n$ and $G \le \Aut(X)$ a subgroup.
Then one of the following properties holds.
\begin{itemize}
\item[(1)]
$G | L_{\C} \ge \Z * \Z$ $($the non-abelian free group of rank two$)$,
and hence $G \ge \Z * \Z$.
\item[(2)]
$G | L_{\C}$ is virtually solvable and
$G \ge K \cap L(\Aut_0(X)) \ge \Z * \Z$ where
$L(\Aut_0(X))$ is the linear part of
$\Aut_0(X)$ {\rm (cf.~\cite[Definition 3.1, p. 240]{Fu})} and
$K = \Ker(G \to \GL(L_{\C}))$, so $X$ is ruled {\rm (cf. \cite[Proposition 5.10]{Fu})}.
\item[(3)]
There is a finite-index solvable subgroup $G_1$ of $G$
such that the null subset $N(G_1)$ of $G_1$
is a normal subgroup of $G_1$ and
$G_1/N(G_1) \cong \Z^{\oplus r}$
for some $r \le n - 1$.
\end{itemize}
In particular, either $G \ge \Z * \Z$ or $G$ is virtually solvable.
In Cases $(2)$ and $(3)$ above, $G | L_{\C}$ is finitely generated.
\end{theorem}

\par \vskip 0.5pc \noindent
{\bf Acknowledgement.} We are very grateful to the referee for suggesting
the implication ``Theorem \ref{Ssimple} $\Rightarrow$ Theorem \ref{Cor3}"
and giving its proof when $N(G) = 1$, as well as lots of constructive suggestions.
The last-named author is supported by an ARF of NUS.

\section{Entropy and Algebraic group action}\label{Pf}

In this section, we shall recall some definitions and technical results
needed in the proofs and establish some easy consequences or already known facts.

\begin{setup} {\bf Terminology and notation}\label{setup2.1}
{\rm
are as in \cite{KM}. Below are some more conventions.

Let $X$ be a compact complex K\"ahler manifold (resp.~a normal projective variety).
As in the introduction,
set $L := H^2(X, \Z)/(\torsion)$ $($resp.~$L := \NS(X)/(\torsion))$,
$L_{\R} := L \otimes_{\Z} \R$, and $L_{\C} := L \otimes_{\Z} \C$.
Let $\widebar{P(X)}$ be the closure of the K\"ahler cone
(resp. the {\it nef cone} $\Nef(X)$, i.e., the closure of the ample cone) of $X$.
Elements in $\widebar{P(X)}$ are called {\it nef}.

For $g \in \Aut(X)$, let
$$d_1(g) := \max\{|\lambda| \, ; \, \lambda \,\, \text{is
an eigenvalue of} \,\, g^* \, | \, H^{1,1}(X) \}$$ be the {\it first dynamical degree} of $g$
(cf.~\cite[\S 2.2]{DS}).
By the generalization of Perron-Frobenius theorem (cf.~\cite{Bi})
applied to $\widebar{P(X)}$,
for every $g \in \Aut(X)$, there is a nonzero nef class $L_g$ (not unique) such that
$$g^* L_g = d_1(g) L_g .$$

We remark that $g$ is of null entropy if and only if so is $g^{-1}$;
if this is the case, then for every non-trivial
$(1, 1)$-class $M$ with $g^* M = \lambda M$, we have
$|\lambda| = 1$.

$G \, | \, Y$ denotes a naturally (from the context) induced action of $G$ on $Y$.
A subvariety $Z \subset X$ is $G$-{\it periodic} if
$Z$ is stabilized by a finite-index subgroup of $G$.
For a complex torus $X$
(as a variety), we have $\Aut_{\variety}(X) = T \rtimes \Aut_{\group}(X)$
where $T = \Aut_0(X)$ $(\cong X$) consists of all the translations of $X$
and $\Aut_{\group}(X)$ is the group of bijective homomorphisms of $X$ (as a torus).

A group $G$ is {\it virtually unipotent}
(resp.\ {\it virtually abelian}, or {\it virtually abelian of rank $r$})
if a finite-index subgroup $G_1$ of
$G$ is unipotent (resp.\ abelian, or isomorphic to $\Z^{\oplus r}$).
A group $G$ is {\it virtually solvable} (resp.\ {\it almost abelian},
or {\it almost abelian of finite rank $r$}, cf.\
definition after \cite[Thm 1.2]{Og07})
if it has a finite-index subgroup $G_1$ and an exact sequence
$$1 \to H \to G_1 \to Q \to 1$$
such that $H$ is finite, and $Q$ is solvable (resp.\ abelian, or
isomorphic to $\Z^{\oplus r}$);
by Lemma \ref{modfin}, $G$ is almost abelian of finite rank $r$ if and only if
$G$ is virtually abelian of rank $r$;
replacing $G_1$ by a finite-index subgroup, we may assume
that
the conjugation action of $G_1$ (and hence of $H$)
on $H$ is trivial;
so in the above definition of virtually solvable group, we may also assume
$H = 1$, so that our definition here coincides with the usual definition.
}
\end{setup}

Theorem \ref{ThB} below follows from Oguiso
\cite[Lemma 2.5]{Og07} and Tits \cite[Theorem 1]{Ti}.

\begin{theorem}\label{ThB}
Let $X$ be a compact K\"ahler $($or projective$)$ manifold
of dimension $n$ and $G$ a subgroup of $\Aut(X)$.
Then we have:
\begin{itemize}
\item[(1)]
Suppose that $G$ is of null entropy. Then $G \, | \, L_{\C}$ is virtually unipotent
and hence virtually solvable $($cf.\ $\ref{setup2.1}$ for notation$)$. Moreover,
$G \,| \, L_{\C}$ is finitely generated.
\item[(2)]
Suppose that $G \, | \, L_{\C} \ge \Z * \Z$.
Then $G$ contains an element of positive entropy.
\end{itemize}
\end{theorem}

Let $X$ be a compact K\"ahler $($resp.\ projective$)$ manifold of dimension $n$.
A sequence $0 \ne L_1 \cdots L_k \in H^{k,k}(X)$ ($1 \le k < n$)
is {\it quasi-nef} if it is inductively obtained in
the following way: first $L_1 \in \widebar{P(X)}$;
once $L_1 \cdots L_{j-1} \in H^{j-1, j-1}(X)$
is defined, we define
$$L_1 \cdots L_j = \lim_{t \to \infty} L_1 \cdots L_{j-1} \cdot M_t$$
for some $M_t \in \widebar{P(X)}$ (cf.\ \cite[\S 2.2]{Z-Tits}).
We remark that for $j \ge 2$,
the $L_j$ which appears
in the construction of the $(j, j)$-class $L_1 \cdots L_j$, may not belong to $\widebar{P(X)}$.
A group $G \le \Aut(X)$ is {\it polarized}
by the quasi-nef sequence $L_1 \cdots L_k$ ($1 \le k < n$)
if $$g^*(L_1 \cdots L_k) = \chi_1(g) \cdots \chi_k(g) (L_1 \cdots L_k)$$
for some characters
$\chi_j : G \to (\R_{> 0}, \times)$.

Theorem \ref{ThC}  gives criteria of virtual solvability,
with (3) proved in \cite[Theorem 1.2]{Z-Tits}.

\begin{theorem}\label{ThC}
Let $X$ be a compact K\"ahler $($resp.\ projective$)$ manifold
of dimension $n$ and $G$ a subgroup of $\Aut(X)$.
Then we have $($cf.\ $\ref{setup2.1}$ for notation of $L_{\C}):$
\begin{itemize}
\item[(1)]
Suppose that $G \, | \, L_{\C}$ is virtually solvable and
its Zariski-closure in $\GL(L_{\C})$
is connected. Then
$G$ is polarized by a quasi-nef sequence $L_1 \cdots L_k$ $(1 \le k < n)$.
\item[(2)]
Conversely, suppose that $G$ is polarized by a
quasi-nef sequence $L_1 \cdots L_k$ $(1 \le k < n)$.
Then $G \, | \, L_{\C}$ is virtually solvable.
\item[(3)]
$G \, | \,  L_{\C}$ is virtually solvable if and only if
there exists a finite-index subgroup $G_1$ of $G$ such that
$N(G_1) \lhd G_1$
and $G_1 / N(G_1) \cong \Z^{\oplus r}$
for some $r \le n-1$.
\end{itemize}
\end{theorem}

We need the following lemmas for the proof of theorems.

\begin{lemma}\label{modfin}
Let $G$ be a group, and $H \lhd G$ a finite normal subgroup.
Then we have:
\begin{itemize}
\item[(1)]
Suppose that for some $r \ge 1$ and $g_i \in G$ we have:
$$G/H = \langle \bar g_1 \rangle \times \cdots \times \langle \bar g_r \rangle
\cong \Z^{\oplus r} .$$
Then there is an integer $s > 0$ such that the subgroup
$G_1 := \langle g_1^s, \dots, g_r^s \rangle$ satisfies
$$G_1 = \langle g_1^s \rangle \times \cdots \times \langle g_r^s \rangle
\cong \Z^{\oplus r}$$
and is of finite-index in $G$; further,
the quotient map $\gamma : G \to G/H$ restricts to an
isomorphism $\gamma \, | \, G_1 : G_1 \to \gamma(G_1)$
onto a finite-index subgroup of $G/H$.
\item[(2)]
A group is almost abelian of finite rank $r$ if and only if it is
virtually abelian of rank $r$.
\end{itemize}
\end{lemma}

\begin{proof}
(2) follows from (1). For (1),
we only need to find $s > 0$ such that $g_i^s$ commutes with $g_j^s$ for all $i, j$.
Since $G/H$ is abelian, the commutator subgroup $[G, G] \le H$.
Thus the commutators $[g_1^t, g_2]$ ($t > 0$) all belong to $H$.
The finiteness of $H$ implies that $[g_1^{t_1}, g_2] = [g_1^{t_2}, g_2]$ for some $t_2 > t_1$,
which implies that $g_1^{s_{12}}$ commutes with $g_2$, where $s_{12} := t_2 - t_1$.
Similarly, we can find an integer $s_{1j} > 0$ such that $g_1^{s_{1j}}$ commutes with $g_j$.
Set $s_1 := s_{12} \times \cdots \times s_{1r}$. Then $g_1^{s_1}$ commutes with
every $g_j$.
Similarly, for each $i$,
we can find an integer $s_{i} > 0$ such that $g_i^{s_i}$ commutes with $g_j$ for all $j$.
Now $s := s_1 \times \cdots \times s_{r}$ will do the job. This proves the lemma.
\end{proof}

\begin{lemma}\label{rad}
Let $X$ be a compact complex K\"ahler manifold $($resp. normal projective variety$)$,
and $G$ a subgroup of $\Aut(X)$.
Then, replacing $G$ by a suitable finite-index subgroup,
the following are true $($cf.~$\ref{setup2.1}$ for notation $L_{\C})$.
\begin{itemize}
\item[(1)]
There is a normal subgroup $U \lhd G$ such that
$U \, | \, L_{\C}$ is unipotent and
$G/U$ is embedded as a Zariski-dense subgroup in a reductive complex linear algebraic group.
\item[(2)]
There is a normal subgroup $R \lhd G$ such that
$R \, | \, L_{\C}$ is solvable and
$G/R$ is embedded as a Zariski-dense subgroup
in a semi-simple complex linear algebraic group $H$.
\end{itemize}
\end{lemma}

\begin{proof}
Let $\overline{G}$ be the Zariski-closure of $G \, | \, L_{\C} \subseteq \GL(L_{\C})$, and
$\iota$ the
composite: $G$ $\to$ $G \, | \, L_{\C} \to \overline{G}$.
Replacing $G$ by the intersection of $G$ and the $\iota$-inverse of
the identity connected component of $\overline{G}$,
we may assume that
$\overline{G}$ is connected. Let $U(\overline{G})$ (resp. $R(\overline{G})$)
be the unipotent radical (resp. the radical) of $\overline{G}$.
Let $U \le G$ (resp. $R \le G$) be the $\iota$-inverse of $U(\overline{G})$
(resp. $R(\overline{G}))$. Then the embeddings $G/U \to \overline{G}/U(\overline{G})$
and $G/R \to \overline{G}/R(\overline{G})$, and $U$ and $R$ here
meet the requirements of the lemma.
\end{proof}

\begin{lemma}\label{G|NS}
We use the notation $L_{\C}$ of $\ref{setup2.1}$. A group $G \le \Aut(X)$ has
finite restriction $G \, | \, L_{\C}$ if and only if the index
$|G : G \cap \Aut_0(X)|$ is finite.
\end{lemma}

\begin{proof}
Consider the exact sequence
$$1 \to K \to G \to G \, | \, L_{\C} \to 1.$$

For an ample divisor or K\"ahler class $\omega$ of $X$, our $K$ is
a subgroup of
$\Aut_{\omega}(X) := \{g \in \Aut(X) \, | \, g^* \omega = \omega\}$,
where the latter contains $\Aut_0(X)$ as a group of finite-index
(cf.\ \cite[Proposition 2.2]{Li}).
Now the last group below is a finite group
$$K / (K \cap \Aut_0(X)) \cong (K \, \Aut_0(X)) / \Aut_0(X)
\le \Aut_{\omega}(X)/\Aut_0(X) .$$
The lemma follows since the connected group $\Aut_0(X)$ acts trivially
on the lattice $L$ (and hence on $L_{\C}$)
so that $G \cap \Aut_0(X) = K \cap \Aut_0(X)$.
\end{proof}

A more precise version of \ref{solext} below was proved in \cite[6.1]{Ca}
for finitely generated groups.

\begin{lemma}\label{solext}
Let $G$ be a group of automorphisms of a compact K\"ahler manifold $X$.
Consider an exact sequence of groups:
$$1 \to N \to G \to Q \to 1 $$
Suppose
$N$ is contained in the union of finitely many connected components of $\Aut(X)$.
Suppose
both $N$ and $Q$ are virtually solvable.
Then $G$ is  also virtually solvable.
\end{lemma}

\begin{proof}
Let $\widebar N \subseteq \Aut(X)$ be the Zariski-closure of $N$.
Replacing $G$ by a suitable finite-index subgroup, we may assume that
$Q$ is solvable.
Since $\widebar N \cap G \, \lhd \, G$, we have
$(\widebar N)_0 \cap G \, \lhd \, G$ for the
identity connected component $(\widebar N)_0$ of $\widebar N$.
Hence $M:= (\widebar N)_0 \cap N \, \lhd \, G$.
Now
$$N/M \cong (N \, (\widebar N)_0)/(\widebar N)_0
\le \widebar N/(\widebar N)_0$$
where the latter is a finite group.
We have an exact sequence
$$1 \to N/M \to G/M \to G/N = Q \to 1 .$$
Replacing $G/M$ by a finite-index subgroup
we may assume that the conjugate action of $G/M$ (and hence of $N/M$) on
the finite group $N/M$ is trivial. Thus
$N/M$ is abelian and hence $G/M$ is solvable.
Since $N$ is virtually solvable so is $\widebar N$.
Hence $(\widebar N)_0$ is solvable. Thus $M$ is solvable.
Therefore, $G$ is solvable.
\end{proof}

\begin{lemma}\label{stab}
Let $X$ be a compact complex K\"ahler manifold $($resp.\ normal projective variety$)$,
and $G \le \Aut(X)$ a subgroup.
Assume the following two conditions:
\begin{itemize}
\item[(1)]
$H \lhd G$; and $H$ has a finite-index subgroup $H_1$ such that
the null set $N(H_1)$ is a $($normal$)$ subgroup of
$H_1$ and $H_1/N(H_1) = \langle \overline{h} \rangle \cong \Z$
for some $h \in H_1$.
\item[(2)]
Suppose that there is a common nef eigenvector $L_1$ of $H_1$,
and further that $h^* L_1 = d_1(h) L_1$,
i.e., $L_1$ equals some $L_h$ up to scalar;
suppose also that for every $s \ne 0$ and every nef $M$
so that $(h^s)^*M = \lambda M$ with $\lambda \ne 1$,
we have $M$ parallel to either one of $L_{h^{\pm1}}$ $($which are two fixed nef eigenvectors$)$.
\end{itemize}
Then the stabilizer subgroup
$\Stab_{L_h}(G) := \{g \in G \, | \, g^*L_h \,\,
\text{is parallel to} \,\, L_h\}$
has index $\le 2$ in $G$.
\end{lemma}

\begin{proof}
We begin with:
\begin{claim}\label{parallel}
For every $g \in G$, the class $g^* L_h$ is parallel to one of $L_{h^{\pm1}}$.
\end{claim}

We prove the claim.
Take any element $g$ of $G$. Since $H \lhd G$, we have $g h g^{-1} \in H \setminus N(H)$. Hence
$g h^a g^{-1}$ is in $H_1$ with $a = |H:H_1|$,
so it equals $h^b n$ for some $b \ne 0$ and $n \in N(H_1)$.
Since the element $n$ of $N(H_1)$ fixes $L_1 = L_h$ (cf.~\ref{setup2.1}),
$$(g h^a g^{-1})^* L_h = d_1(h)^b L_h, \,\,\,\, (h^a)^*(g^* L_h) = d_1(h)^b (g^* L_h) .$$
Now the condition (2) applied to $h^a$ and $\lambda := d_1(h)^b \ne 1$,
implies the claim.

Return to the proof of Lemma \ref{stab}.
Suppose there is some $g_1 \in G \setminus \Stab_G(L_h)$.
Take any $g \in G \setminus \Stab_G(L_h)$. By Claim \ref{parallel}, we have
(with equalities all up to scalars)
$g^* L_h = L_{h^{-1}} = g_1^* L_h$, $(g g_1^{-1})^* L_h = (g_1^{-1})^* g^* L_h = L_h$.
Hence $g g_1^{-1} \in \Stab_G(L_h)$ and $g = (g g_1^{-1}) g_1 \in \Stab_G(L_h) \, g_1$.
So $G = \Stab_G(L_h) \, \cup \, \Stab_G(L_h) \, g_1$.
The lemma follows.
\end{proof}

\begin{setup}
{\bf Proof of Theorem \ref{ThB}}
\end{setup}

Assertion (2) follows from Assertion (1),
so we only need to prove Theorem \ref{ThB}(1).

We follow the proof of Oguiso \cite[Prop 2.2]{Og07}. Since
$G$ is of null entropy, the subset
$$U := \{g \in G \, ; \, g \, | \, L_{\C} \,\,\, \text{is unipotent}\}$$
is a normal subgroup of $G$.
If $G \, | \, L_{\C}$
is not virtually solvable, then by the
classical Tits alternative theorem \cite[Theorem 1]{Ti},
there are $g_i \in G$ such that $\langle g_1, g_2 \rangle | L_{\C}
= (\langle g_1 \rangle | L_{\C}) * (\langle g_2 \rangle | L_{\C}) = \Z * \Z$.
As observed in \cite{Og07}, $g_i^s \in U$ for some $s \ge 1$, and hence
$\Z * \Z = \langle g_1^s, g_2^s \rangle \, | \, L_{\C} \le U \, | \, L_{\C}$
which is unipotent (and hence solvable). This is absurd.

Thus, $G \, | \, L_{\C}$ is virtually solvable.
Replacing $G$ by a suitable finite-index subgroup, we may assume that
$G \, | \, L_{\C}$ is solvable and its closure $\widebar G$ in
$\GL(L_{\C})$ is connected (and solvable).
Write $\widebar G = \widebar U \rtimes \widebar T$
where $\widebar U$ is the unipotent radical and
$\widebar T$ a maximal torus in $\widebar G$.
As observed in \cite{Og07}, the image of $G$ via the quotient map
$\widebar G\to \widebar T$ is a torsion group in $\GL(L_{\C})$
with bounded exponent
and hence a finite group by Burnside's theorem.
Thus the index $|G:U| < \infty$.

To finish the proof of Assertion (1),
we may assume that $G = U$ and it suffices to show that
$G \, | \, L$ is generated by $\ell (\ell -1)/2$ elements
where $\ell = \rank L$. Regarding $\widebar G$ as a subgroup of upper triangular
matrices, there is a standard
normal series
$$1 \lhd U_{1} \lhd U_2 \lhd \cdots \lhd U_{\ell (\ell -1)/2} = \widebar G$$
such that the factor groups are all $1$-dimensional.
Restricting the series to $G \, | \, L$, we get a normal series of discrete
groups whose factor groups are cyclic groups.
Thus $G \, | \, L$ is generated by $\ell (\ell -1)/2$ elements.
This proves Theorem \ref{ThB}.

\begin{setup}
{\bf Proof of Theorem \ref{ThC}}
\end{setup}

(1) was proved in \cite[Theorem 1.2]{Z-Tits}.
For (2), suppose that $G$ is polarized by a quasi-nef sequence
$L_1 \cdots L_k$ ($1 \le k < n$) so that
$g^*(L_1 \cdots L_k) = \chi_1(g) \cdots \chi_k(g) L_1 \cdots L_k$.
As in the proof of \cite[Theorem 1.2]{Z-Tits}, the homomorphism
$$
\varphi : G \to (\R, +), \,\,\,\,
g \mapsto (\log \chi_1(g), \dots, \log \chi_{n-1}(g))
$$
has $\Ker(G) = N(G)$, and $\varphi(G) = \Z^{\oplus r}$ a lattice
in $\R^{n-1}$.
By Theorem \ref{ThB}, $N(G) \, | \, L_{\C}$ is virtually solvable,
so is $G \, | \, L_{\C}$, since $G/N(G)$ is abelian and by Lemma \ref{solext}.

For (3), the ``if part" follows from Theorem \ref{ThB} and Lemma \ref{solext}.
The ``only if" part is by \cite[Theorem 1.2, Remark 1.3]{Z-Tits}.
This proves Theorem \ref{ThC}.

\begin{setup}
{\bf Proof of Theorem \ref{ThA}}
\end{setup}

We may assume that Assertion (1) is not satisfied.
Replacing $G$ by a suitable finite-index subgroup and by \cite[Thm 1]{Ti},
we may assume that $G | L_{\C}$ is
solvable and its closure $\widebar G$ in $\GL(L_{\C})$
is connected (and solvable).
Let $K = \Ker(G \to G | L_{\C})$
be as in Lemma \ref{G|NS}.

Suppose that $K$ is virtually solvable. Then so is $G$
by Lemma \ref{solext}. Thus Theorem \ref{ThA}(3) occurs,
by \cite[Theorem 1.2, Remark 1.3]{Z-Tits}.

Suppose that $K$ is not virtually solvable.
Consider the exact sequence
$$1 \to L_A \to \Aut_0(X) \to T \to 1$$
where $L_A$ is the linear part of $\Aut_0(X)$ and $T$ a compact complex torus
(cf.\ \cite[Theorem 3.12]{Li}).
This induces the exact sequence (with $Q$ abelian):
$$1 \to K \cap L_A \to K \cap \Aut_0(X) \to Q \to 1 .$$

We may assume that Theorem \ref{ThA}(2) does not occur.
So $K \cap L_A$ is virtually solvable by Tits alternative.
Thus so is $K \cap \Aut_0(X)$ by the exact sequence
above and Lemma \ref{solext}.
Now
$K / (K \cap \Aut_0(X))$ is a finite group by Lemma \ref{G|NS}.
Hence $K$ is also virtually solvable, contradicting our extra assumption.

For the final assertion, replacing $G$ by a suitable finite-index subgroup,
we may assume that
$G \, | \, L_{\C}$ is solvable and has connected Zariski-closure
in $\GL(L_{\C})$ .
Then $G/N(G) \cong \Z^{\oplus r}$ by \cite[Theorem 1.2]{Z-Tits}. This
and Theorem \ref{ThB} for $N(G)$ imply the assertion.

\section{Strong primitivity for threefolds}\label{sp}

{\it We prove Theorem} \ref{ThF}.
Replacing $G_0$ by the identity connected component of its
Zariski-closure in $\Aut_0(X)$
we may further assume that $G_0 = G \cap \Aut_0(X)$ is
connected, positive-dimensional and closed in $\Aut_0(X)$.
Because $G$ acts naturally on the quotient of $X$ modulo $G_0$ and because of
our assumption, we may assume that one orbit of $G_0$ is a
Zariski-dense open subset of $X$,
i.e., $X$ is almost homogeneous
(cf. \cite[Lemma 2.14]{Z-Tits}).

\begin{claim}\label{c0ThF}
Suppose the irregularity $q(X) = h^1(X, \OO_X) > 0$. Then Theorem $\ref{ThF}$ is true.
\end{claim}

We prove Claim \ref{c0ThF}.
By the proof of \cite[Lemma 2.13]{Z-Tits}, the albanese map
$a: X \to A:= \Alb(X)$ is surjective, birational and
necessarily $\Aut(X)$-equivariant. Our $G_0$ induces an action on $A$
and we denote it by $G_0 | A$.
Since $G_0 | A$ also has a Zariski-dense open orbit in $A$,
we have $G_0 | A = \Aut_0(X)$ ($\cong A$).
Let $B \subset A$ be the locus over which $a$ is not an isomorphism.
Note that $B$ and $a^{-1}(B)$ are $G_0$-stable.
Since $G_0 | A = \Aut_0(X)$, we have $B = \emptyset$.
Claim \ref{c0ThF} is proved.

\par \vskip 1pc
We continue the proof of Theorem \ref{ThF}.
By Claim \ref{c0ThF}, we may assume that $q(X) = 0$.
Thus $G_0 \le \Aut_0(X)$ is a linear algebraic group
and has a Zariski-dense open orbit in $X$.
In particular, $X$ is ruled and unirational,
because linear algebraic groups are rational
varieties by a classical result of Chevalley.

In the rest of the proof, we shall derive a contradiction.
Let $U \subseteq X$ be the open dense $G_0$-orbit
and $F := X \setminus U$. Then $F$ consists of finitely many prime divisors
and some subvarieties of codimension $\ge 2$.
Since $G_0 \lhd G$, we may assume that both $U$ and
all irreducible components of $F$
are $G$-stable, after replacing $G$ by a suitable finite-index subgroup.
$X$ has only finitely many $G_0$-periodic prime divisors,
all of which are contained in $F$ and $G$-stable.

By the minimal model program (MMP) in dimension three (cf. \cite[\S 3.31, \S 3.46]{KM}),
the end product of a uniruled variety (like our $X$ here) is
an extremal Fano contraction $f: X_m \to Y$ with a general fibre $X_{m,y}$, i.e., by definition,
the restriction $-K_{X_m} \, | \, X_{m, y}$ of the canonical divisor $-K_{X_m}$
is ample and the Picard numbers satisfy $\rho(X_m) = 1 + \rho(Y)$.

\begin{claim}\label{c1ThF}
\begin{itemize}
\item[(1)]
Every $G_0$-periodic subvariety of $X$ is actually $G_0$-stable.
\item[(2)]
There are a composite $X = X_0 \dasharrow X_1 \cdots \dasharrow X_m$
of birational extremal contractions and an extremal Fano contraction
$X_m \to Y$  with $\dim Y < \dim X$.
The induced birational action of $G_0$ on each $X_i$
is biregular. $G_0 | X_m$ descends to an action on $Y$ so that
$X_m \to Y$ is $G_0$-equivariant.
\item[(3)]
In $(2)$, for every finite-index subgroup $G_1$ of $G$,
there is at least one $i \in \{1, \dots, m\}$
such that the induced action of $G_1$ on $X_i$ is not biregular.
\item[(4)]
In $(2)$, let $s \le m$ be the largest integer such that
$X_i \to X_{i+1}$ is divisorial for
every $i \in \{0, 1, \dots, s-1\}$.
Then, replacing $G$ by a suitable finite-index subgroup,
the induced birational action of $G$ on each $X_i$ $(i < s)$ is biregular and hence
each map $X_{i-1} \to X_{i}$ is $G$-equivariant.
In particular, $s < m$.
\end{itemize}
\end{claim}

\begin{remark}
By the choice of $s$ in
(4), $X_s \dashrightarrow X_{s+1}$ is a flip
with a flipping contraction $X_s \to Y_s$
and with $X_{s+1} = \Proj_{Y_s}(\oplus_{m \ge 0} \OO_{Y_s}(mK_{Y_s}))$
(cf.~\cite[Cor 6.4 or Thm 3.52]{KM}).
\end{remark}

We now prove Claim \ref{c1ThF}.
(1) is true because $G_0$ is a connected group.
For the first part of (2), see \cite[\S 3.31, \S 3.46]{KM} when $\dim X = 3$ and
\cite[Corollary 1.3.2]{BCHM} when $\dim X$ is arbitrary.
The second part of (2) is true because
$G_0$ acts trivially on $H^i(X, \Z)$,
and also on $\NS_{\C}(X)$ and the extremal rays of $\NE(X)$
(cf.\ \cite[Lemmas 2.12 and 3.6]{uniruled}).

For (4),
suppose that $X = X_0 \to X_1$ is a divisorial contraction
of an extremal ray $R := \R_{> 0} [\ell]$ with
an
exceptional divisor $D_0$.
Since $G_0$ acts trivially on the extremal rays of $\NE(X)$,
this $D_0$ is $G_0$-stable. So $D_0$ is contained in $F$
and is $G$-stable.

Since
the natural map $G/G_0 \to \Aut(H^2(X, \Z))$ has finite kernel
(cf. \cite[Proposition 2.2]{Li}) and
by the assumption, $G/G_0$ is almost abelian of finite rank $r > 0$.
By Lemma \ref{modfin} and
replacing $G$ by a finite-index subgroup,
there is some $H_0 \lhd G$ such that $H_0$ contains
$G_0$ as a subgroup of finite index and
$G/H_0 = \langle \bar{g}_1 \rangle \oplus \cdots \oplus \langle \bar{g}_r \rangle
\cong \Z^{\oplus r}$ for some $g_i \in G$.

In \cite[Lemma 3.7]{uniruled}, it is proved that a positive power
of $g_i$ preserves the extremal ray $R$ and hence descends to a
biregular automorphism of $X_1$.
Thus $X \to X_1$ is $G$-equivariant
after $G$ is replaced by a finite-index subgroup.
Indeed, we may assume that $g_i(R) = R$ so that
$\{g(R) \, | \, g \in G\}$ consists of no more than $|H_0 : G_0|$
extremal rays so that a finite-index subgroup of $G$ fixes $R$.
{\it This and the second sentence of the next paragraph are the places
where we need $G/G_0$ to be almost abelian.}

For (3), suppose the contrary that $G$ (replaced by a finite-index subgroup)
acts biregularly on all $X_i$.
Then, as in the proof of (4) above,
by \cite[Theorem 2.13, or Appendix]{uniruled},
we may assume that
$G$ (replaced by its finite-index subgroup) fixes the extremal ray
giving rise to the extremal Fano contraction $X_m \to Y$, and hence
$X_m \to Y$ is $G$-equivariant.
By the strong primitivity assumption, we have $\dim Y = 0$, so the Picard number
$\rho(X_m) = 1$ and $-K_{X_m}$ is ample.
Since $G$ fixes the ample class of $-K_{X_m}$,
it is a finite extension of $G_0$ (cf.\ \cite[Proposition 2.2]{Li}).
This contradicts the assumption.
Claim \ref{c1ThF} is proved.

\begin{claim}\label{c2ThF}
It is impossible that $\NS_{\C}(X_i)$ with $0 \le i \le m$ is spanned by $-K_{X_i}$ and
$G_0$-periodic divisors, or that $\NS_{\C}(Y)$ is spanned by
$G_0$-periodic divisors.
\end{claim}

Indeed, note that $\NS_{\C}(X_m)$ is spanned by
$-K_{X_m}$ (which is ample over $Y$)
and the pullback of $\NS_{\C}(Y)$, and $\NS_{\C}(X)$
is spanned by the pullback of $\NS_{\C}(X_i)$ and (necessarily $G_0$-stable)
exceptional divisors of $X \dasharrow X_i$.
Thus we only need to rule out the possibility that
$\NS_{\C}(X)$ is spanned by $-K_X$, and $G_0$-stable divisors $D_i$
all of which are necessarily contained in $F$ and hence $G$-stable.

Write an ample divisor $M$ on $X$ as a combination
of $-K_X$ and $D_i$'s. Then $G \le \Aut_{[M]}(X)$,
so $|G/G_0| < \infty$ as in the proof of Lemma \ref{G|NS},
contradicting the assumption.

\begin{claim}\label{c3ThF}
$X_m$ and hence $Y$ contain a $G_0$-fixed point
$($here we use that $\dim X = 3)$.
\end{claim}

Indeed, note that a smooth threefold has no flip and a flip preserves
the singularity type of a threefold.
By Claim \ref{c1ThF}, for some $m-1 \ge t \ge s$,
$X_t \dasharrow X_{t+1}$ is a flip
and $X_{t+1} \to \cdots \to X_m$ is the composite of
extremal divisorial contractions. So the non-empty finite set
$\Sing X_{t+1}$ (cf. \cite[Corollary 5.18]{KM}) and its image on $X_m$ are fixed by $G_0$.

\begin{claim}\label{c4ThF}
It is impossible that $\dim Y \le 1$.
\end{claim}

Indeed, if $\dim Y = 0$,
then $\NS_{\C}(X_m)$ is of rank one and spanned by
$-K_{X_m}$ (which is ample over $Y$).
This contradicts Claim \ref{c2ThF}.
If $\dim Y = 1$
then the rank two space $\NS_{\C}(X_m)$ is spanned by $-K_{X_m}$
(which is ample over $Y$) and
the fibre over a $G_0$-fixed point $y_0$
(cf.\ Claim \ref{c3ThF}).
This contradicts Claim \ref{c2ThF}.

\par \vskip 1pc
We continue the proof of Theorem \ref{ThF}.
Take an extremal ray on $X_s$ generated by a rational curve $\ell$
and let $X_s \dasharrow X_{s+1}$ be the flip (cf.\ Claim \ref{c1ThF} for $s$).
Note that $G_0$ stabilizes all irreducible components $E_i$ of
the exceptional locus of the flipping contraction $X_s \to Y_s$
and $G$ (replaced by a finite-index subgroup) stabilizes
all irreducible components $D_{ij}$
of the Zariski closure of $\cup_{g \in G} \, g(E_i)$, because $G_0 \lhd G$.
These $D_{ij}$ are unions of `small' $G_0$-orbits and
hence are contained in the image of the algebraic subset $F \subset X$.

If $\dim D_{ij} = \dim E_i = 1$, then
$G$ preserves the extremal ray $\R_{\ge 0}[\ell] \subseteq \NE(X_s)$ and
we can descend $G$ to a biregular action on $X_{s+1}$
(cf.\ \cite[Lemma 3.6]{uniruled}).
Now apply MMP on $X_{s+1}$ and continue the process.

Assume that $\dim D_{ij} = 2 > \dim E_i = 1$.
If $G_0$ acts trivially on some $g_0E_i$ in
the set $\{g E_i \, | \, g \in G\}$,
then $G_0 = gG_0g^{-1}$ acts trivially on $gg_0 E_i$, i.e., on
all $g'E_i$ ($g' \in G$).
Hence $G_0 \, | \, D_{ij} = \id$.
This contradicts Claim \ref{c5ThF} below.

Suppose that $G_0$ acts non-trivially on some $g_0E_i$
and hence on all $gE_i$ ($g \in G$).
Then these extremal curves $gE_i$ are fibres of the quotient map
$D_{ij} \to D_{ij}/G_0 =: B$
over a curve $B$, hence homologous to each other.
So they give rise to one and the same class in
the extremal ray $\R_{\ge 0}[\ell] \subseteq \NE(X_s)$.
Thus $G$ preserves this extremal ray and
we can descend $G$ to a biregular action on $X_{s+1}$
(cf.\ \cite[Lemma 3.6]{uniruled}).
Now apply MMP on $X_{s+1}$ and continue the process.
Therefore, we can continue the $G$-equivariant MMP and
reach an extremal Fano fibration
$X_m \to Y$ which is a contradiction (cf.\ Claim \ref{c1ThF}).

To complete the proof of Theorem \ref{ThF}, we still need to prove:

\begin{claim}\label{c5ThF}
It is impossible that $\dim Y = 2$, $\dim D_{ij} = 2$ and $G_0 \, | \, D_{ij} = \id$.
\end{claim}

We now prove Claim \ref{c5ThF}.
$X_m \to Y$ is known as an extremal conic fibration.
We may take a $G_0$-equivariant blowup $\hat{X} \to X_s$ to resolve indeterminacy of the composite
$\pi_s : X_s \dasharrow X_m \to Y$ so that the induced map
$\hat{\pi}: \hat{X} \to Y$ is holomorphic and $G_0$-equivariant.
By \cite[Theorem 4.8]{Mi}, there exist blowups
$\sigma_x: X' \to \hat{X}$ and $\sigma_y : Y' \to Y$ with
$X'$ and $Y'$ smooth, and
extremal conic fibration $\pi' : X' \to Y'$ such that
$\hat{\pi} \circ \sigma_x = \sigma_y \circ \pi'$.
We may also assume that the four maps above are $G_0$-equivariant
by taking extra blowups so that they are equivariant
(noting that $G_0$ stabilizes extremal rays).

We will reach a contradiction to Claim \ref{c2ThF}.
To do so, we consider both $Y$ and $Y'$.

Indeed, if $K_{Y'}^2 \le 7$, then $\NS_{\C}(Y')$ (and hence $\NS_{\C}(Y)$) are spanned by
$G_0$-stable curves (i.e., the negative curves on $Y'$).
This contradicts Claim \ref{c2ThF}.
Therefore, we may assume that $K_{Y'}^2 = 9$ or $8$, and
$Y'= \BPP^2$ or a Hirzebruch surface $F_d$
of degree $d \ge 0$.

If $Y' = \BPP^2$ or $Y' = \BPP^1 \times \BPP^1$, then $Y'$
has no negative curve to contract, so
$Y' = Y$.

If $Y = F_d$ then $G_0$ stabilizes a fibre passing through a
fixed point of $y_0$ of $G_0 | Y$
(cf.\ Claim \ref{c3ThF}), and the zero-section through $y_0$
(resp.\ the unique $(-d)$-curve) when $d = 0$ (resp.\ $d \ge 1$).
This contradicts Claim \ref{c2ThF}.

Therefore, we may assume that either $Y = Y' = \BPP^2$, or
$F_d = Y' \to Y$ (with $d \ge 1$) is the contraction of
the unique $(-d)$-curve. Thus the Picard number $\rho(X_m) = 1 + \rho(Y) = 2$.

Let $D_{ij}' \subset X'$ be the proper transform of $D_{ij} \subset X_s$.
Then $G_0$ acts trivially on $D_{ij}'$ because so does $G_0$ on $D_{ij}$.
Since every fibre of $\pi' : X' \to Y'$ is $1$-dimensional,
the image $C_{ij} \subseteq Y'$ of $D_{ij}'$ is the whole $Y'$ or a curve,
and $G_0 \, | \, C_{ij} = \id$. Since $G_0 \, | \, Y = \id$ would contradict
Claim \ref{c2ThF}, we may assume that $C_{ij}$ is a curve in $Y'$.
If $Y'= Y = \BPP^2$ (resp.\ $Y' = F_d \to Y$ is the contraction of
the $(-d)$-curve),
then $G_0 | Y$ stabilizes $C_{ij}$ (resp.\ the image of
$C_{ij}$ or every generating line).
This contradicts Claim \ref{c2ThF}. Claim \ref{c5ThF} is proved.

\begin{corollary}\label{Cor1}
Let $X$ be a $3$-dimensional normal projective variety
and $G \le \Aut(X)$ a subgroup of {\rm null entropy}
such that $G_0 := G \cap \Aut_0(X)$ is infinite and
the quotient $G/G_0$ is
an almost abelian group of positive rank.
Then $(X, G)$ is not strongly primitive.
\end{corollary}

\begin{proof}
Taking a $G$-equivariant resolution, we may assume that $X$ is smooth.
With our assumption and the proof of Theorem \ref{ThF},
we may assume that $X$ is a complex torus, and $G_0$ is
connected, is closed and has a Zariski-dense open orbit in $X$.
Thus $G_0 = \Aut_0(X)$.
By Lemma \ref{modfin} and replacing $G$ by a suitable finite-index subgroup,
we may assume that $G/G_0$ is abelian and equals
$\langle \bar{g}_1, \dots, \bar{g}_r \rangle$
for $g_i \in G$, where the order $o(\bar{g}_i) = \infty$;
moreover, $g_i$ has unipotent representation matrix on $H^0(X, \Omega_X^1)$
using Kronecker's theorem as in \cite[Lemma 2.14]{JDG}.
Write $g_i = T_{t_i} \circ h_i$ where $T_{t_i}$ is the translation by $t_i$
and $h_i$ is a group automorphism.
As in \cite[Lemma 2.15]{JDG}, the identity connected component
$B$ of the fixed locus
$X^{h_1}$ (pointwise) has dimension
equal to that of $\Ker(h_1^* - \id) \subset H^0(X, \Omega_X^1)$
and is hence between $1$ and $\dim X - 1$.
Note that $h_1h_j = h_jh_1$ holds modulo $\Aut_0(X)$ and hence holds in $\Aut(X)$
since both sides fix the origin. Thus $h_j(B)$ is contained in $X^{h_1}$
and hence equals $B$ since it contains the origin.
Now $g_j(x + B) = g_j(x) + B$. So $g_j$ permutes cosets of the quotient torus $X/B$;
the same is true for elements of $\Aut_0(X)$.
Thus, the quotient map $X \to X/B$ is $G$-equivariant.
This proves Corollary \ref{Cor1}.
\end{proof}

\begin{setup}
{\bf Proof of Corollary \ref{Cor2}}
\end{setup}

Taking a $G$-equivariant resolution, we may assume that $X$ is smooth.
If $q(X) > 0$, then, by Claim \ref{c0ThF}, we may
assume that $X$ is a complex torus so that $\Aut_0(X) \ne \{\id_X\}$.
Thus, we may always assume that $\Aut_0(X) \ne \{\id_X\}$.

Replacing $G$ by $G . \Aut_0(X)$ we may assume that $G \ge G_0 := \Aut_0(X)$.
According to Lemma \ref{G|NS}, we have
$G | \NS_{\C}(X) = G/K$ where $|K/G_0| < \infty$.
Thus $G/G_0$ is also almost abelian of positive rank by assumption.
Now Corollary \ref{Cor2} follows from Corollary \ref{Cor1}.

\section{Minimal threefolds}\label{mt}

Below sufficient conditions for being a quotient of a torus are given.

\begin{theorem}\label{Torus}
Let $X$ be a $3$-dimensional minimal projective variety.
Assume that one of the following two properties is satisfied:
\begin{itemize}
\item[(1)]
The first Chern class $c_1(X) = 0$.
The second Chern class $c_2(X)$ $($as a linear form on $\NS_{\C}(X)$ as in \cite[p. 265]{SW}$)$
has zero intersection with a nef and big $\R$-divisor.
\item[(2)]
There is a subgroup $G \le \Aut(X)$ such that the null set
$N(G)$ is a subgroup of $G$ and $G/N(G) \cong \Z^{\oplus 2}$.
\end{itemize}
Set $B := \Aut(X)$.
Then there is a $B$-equivariant birational surjective morphism $X \to X'$
such that $X' \cong T/F$ for a finite group $F$ acting freely outside a finite set of
an abelian variety $T$ of dimension three. Further, the action of $B$ on $X'$
lifts to an action of a group
$\widetilde{B}$ on $T$ such that $\widetilde{B} / F \cong B$.
\end{theorem}

\begin{proof}
Assume the condition (1) in Theorem \ref{Torus}.
Let
$$D := \Nef(X) \cap c_2(X)^{\perp} = \{M \in \Nef(X) \, | \, M . c_2(X) = 0\}$$
be a closed subcone of the nef cone $\Nef(X)$ of $X$.
Let
$$C := \NE(X) \cap D^{\perp} = \{[\ell] \in \NE(X) \, | \, \ell . D_0 = 0 \,\, \text{for all} \,\, D_0 \in D\}$$
be a closed subcone of the closed cone $\NE(X)$ of effective curves on $X$.
Then $c_2(X) \in C$ by definition and using Miyaoka's pseudo-effectivity of $c_2$
for any minimal variety $X$ of dimension $n$:
$c_2(X) \cdot (H_1 \cdots H_{n-2}) \ge 0$ for all nef divisors $H_i$ on $X$
(cf. \cite[Theorem 4.1, Proposition 1.1]{SW}).

By assumption, $D$ contains a nef and big $\R$-divisor.
Let $A$ be an interior element of $D$.
As in \cite[Theorem 3.9.1]{BCHM}, there is a birational contraction
$$\sigma: X \to X'$$
such that a curve $\ell \subset X$ is contracted to
a point if and only if the class $[\ell]$ is contained in $C$,
and such that $A = \sigma^*A'$ for some ample $\R$-divisor $A'$.

By the projection formula and since $A$ is contained in $D$, we
have $A' . c_2(X') = \sigma^*A'. c_2(X)$
$= A . c_2(X)= 0$.
For any ample $\R$-divisor $P$ on $X'$, a small perturbation
$A_{\varepsilon}' := A' - \varepsilon P$ of the ample divisor $A'$ is
still ample because the ample cone of $X'$ is open.
By Miyaoka's pseudo-effectivity of $c_2$ for minimal variety,
we have
$$0 \le  \varepsilon P . c_2(X') \le (A_{\varepsilon}' + \varepsilon P) . c_2(X') = A' . c_2(X') = 0 .$$
So $P . c_2(X') = 0$. Since $\NS_{\C}(X')$ is spanned by ample divisors,
we obtain then $c_2(X') = 0$ as a linear form on $\NS_{\C}(X')$.

Thus, $c_1(X)$ and $c_2(X)$ vanish, and by \cite[Corollary, p. 266]{SW}, we have $X' = T/F$
where $F$ is a finite group acting on the abelian variety $T$ freely outside
a finite set.
Since $D$ and hence $C$ are stable under the action of $B := \Aut(X)$, the
contraction $\sigma : X \to X'$ is $B$-equivariant.
By \cite[\S 3, especially Proposition 3]{Be} applied to \'etale-in-codimension-one covers,
replacing $T$ by the finite cover corresponding to the maximal lattice
in $\pi_1(X' \setminus \Sing X')$,
we can lift the action of $B$ on $X'$ to an action of
a group $\widetilde{B}$ on $T$ such that $\widetilde{B}/\Gal(T/X') \cong B$.
This proves Theorem \ref{Torus} under condition (1).

Next, assume condition (2) in Theorem \ref{Torus}.
The maximality of the rank of $G/N(G)$ and \cite[Lemma 2.11]{Z-Tits}
imply the Kodaira dimension $\kappa(X) = 0$.
The abundance theorem for minimal threefolds implies $K_X \sim_{\Q} 0$
(cf.~\cite[3.13]{KM}).
Replacing $G$ by a finite-index subgroup, we may assume that $G \, | \, \NS_{\C}(X)$ is
solvable and has connected Zariski-closure in $\GL(\NS_{\C}(X))$
(cf.~Theorem \ref{ThB} or \ref{ThC}).
By \cite[Claim 2.5(1)]{CY3}, $c_2(X)$ is perpendicular to a nef and big $\R$-divisor.
We are reduced to condition (1).
This proves Theorem \ref{Torus}.
\end{proof}

\par \vskip 0.5pc
The next is the key step towards Theorem \ref{Cor3}.

We recall the notation in the Introduction: For a subgroup $G$ of $\Aut(X)$,
let $\overline{G} \subseteq \GL(\NS_{\C}(X))$ be the Zariski-closure
of $G \, | \, \NS_{\C}(X)$ and $R(\overline{G})$ its solvable radical,
both of which are defined over $\Q$.
We have
a natural composition of homomorphisms:
$$\iota: G \to G \, | \, \NS_{\C}(X) \to \overline{G}.$$

\begin{theorem}\label{Ssimple}
Let $X$ be a $3$-dimensional minimal projective variety and
$G \le \Aut(X)$ a subgroup
such that $G \, | \, \NS_{\C}(X)$ is not virtually solvable.
Then we have:

\begin{itemize}
\item[(1)]
Suppose that
$R(G) := \iota^{-1}(\iota(G) \, \cap \, R(\widebar{G}))$ is of null entropy.
Then $R(G) \, | \, \NS_{\C}(X)$ is virtually unipotent and hence of null entropy.
Replacing $G$ by a suitable finite-index subgroup,
$G/R(G)$ is embedded as a Zariski-dense subgroup in
$H := \overline{G}/R(\overline{G})$ so that $H(\R)$ is
a semi-simple real linear algebraic group
and is either of real rank $1$ {\rm (cf. \cite[0.25]{Ma})}
or locally isomorphic to $\SL_3(\R)$ or $\SL_3(\C)$.

\item[(2)]
Suppose that $R(G)$ is not of null entropy.  Set $B := \Aut(X)$.
Then there is a $B$-$($and hence $G$-$)$ equivariant birational surjective morphism $X \to X'$
such that $X' \cong T/F$ for a finite group $F$ acting freely outside a finite set of
an abelian variety $T$ of dimension three.
Further, the action of $B$ on $X'$ lifts to an action of a group
$\widetilde{B}$ on $T$ such that $\widetilde{B} / F \cong B$.
\end{itemize}
\end{theorem}

{\it We now prove Theorem} \ref{Ssimple}.
Note that $\iota(G)$ is contained in $\overline{G}(\Q)$, and
is Zariski-dense in $\overline{G}$.
As in Lemma \ref{rad}, replacing $G$ by a suitable finite-index subgroup,
we may assume
$\overline{G}$ is connected.
Set $R := R(G)$. The $\iota: G \to \overline{G}$ above
induces an injective homomorphism:
$$\gamma: G/R \to H := \overline{G}/R(\overline{G}) .$$
Indeed, $\gamma$ is defined over $\Q$
(cf.~\cite[0.11]{Ma}).
Of course, $H$ is semi-simple.
$R \, | \, \NS_{\C}(X)$ is solvable, being embedded in the solvable group $R(\overline{G})$.

\begin{lemma}
Up to finite index, $H(\R)$ is either semi-simple and of real rank one, or
locally isomorphic to $\SL_3(\R)$ or $\SL_3(\C)$.
\end{lemma}

\begin{proof}
Let $S$ be a Levi subgroup of $\overline{G}$ such that $\overline{G} = R(\overline{G})S$.
Our $S$ can be chosen to be defined over $\Q$ and the induced composite homomorphism
$S \to \overline{G} \to H$ is a $\Q$-isogeny (cf. \cite[Proof of Proposition 11.23]{Bo}).
Now the argument in \cite[Theorem 5.1, Proposition 5.2]{CZ} for the action
$S(\R) \, | \, \NS_{\R}(X)$ ($\Q$-isogeny to $H(\R)$)
implies that either $H(\R)$ is of real rank $\le 1$, or
$H(\R)$ is of real rank $\ge 2$ and is locally
isomorphic to $\SL_3(\R)$ or $\SL_3(\C)$.
In fact, our $S(\R)$ and indeed even the larger group $\overline{G}(\R)$
already act on $\NS_{\R}(X)$
as the extension of the geometrically induced action of the Zariski-dense
subgroup $G \, | \, \NS_{\R}(X)$ of $\overline{G}(\R)$, so we do not need Margulis'
condition on $G \, | \, \NS_{\R}(X)$ there, for the extension of the action.
To be precise, one main purpose of the extra assumption in \cite{CZ} on the rank of a lattice
(acting on $X$) of a semi-simple
real Lie group is to extend the action of the lattice on cohomology groups of $X$
to an action of the real Lie group.

If $H$ has real rank $\rk_{\R} H = 0$, then $H(\R)$ is compact. The image of $G$ in
$H(\R)$ is contained in an arithmetic subgroup of $H$, hence
discrete and finite. This image is Zariski-dense in $H(\R)$.
Thus $H(\R)$ is a finite group and hence $G$ is a finite extension of $R$,
so $G \, | \, \NS_{\C}(X)$ is virtually solvable, which contradicts the assumption.
Thus, $\rk_{\R} H$ is at least one and the lemma is proved.
\end{proof}

We return to the proof of Theorem \ref{Ssimple}.
Suppose that $R$ is of null entropy.
As mentioned in the proof of Theorem \ref{ThB} (cf.~\cite[Proposition 2.2]{Og07})
the set
$$U(R) := \{g \in R \, ; \, g^* \, | \, \NS_{\C}(X) \,\, \text{is unipotent}\}$$
is a finite-index subgroup of $R$.
So Theorem \ref{Ssimple}(1) occurs
by the lemma above.

Hence we may assume that $R$ is not of null entropy.
We will deduce Theorem \ref{Ssimple}(2).
Take $R_1 \le R$ a finite-index subgroup such that
$R_1 \, | \, \NS_{\C}(X)$
has connected Zariski-closure in $\GL(\NS_{\C}(X))$ (and is solvable).
Thus $R_1/N(R_1) \cong \Z^{\oplus r}$ for some $1 \le r \le \dim X -1 = 2$
by \cite[Theorem 1.2]{Z-Tits}.
If $r = 2$, then Theorem \ref{Ssimple}(2) holds by Theorem \ref{Torus}.

Thus we may assume $r = 1$, i.e., $\langle \overline{h} \rangle = R_1/N(R_1) \cong \Z$
with $h \in R_1$ of positive entropy.

\begin{lemma}\label{K0}
$K_X \sim_{\Q} 0$.
\end{lemma}

\begin{proof}
Suppose the contrary that the lemma is false.
By the $3$-dimensional minimal model program and abundance theorem (cf.~\cite[3.13]{KM}),
the Kodaira dimension $\kappa(X)$ is positive, and $|m K_X|$ (for some $m > 0$)
defines a holomorphic
map $\psi : X \to Y$ with connected fibres and $\dim Y = \kappa(X)$.
The induced action
of $G$ on $Y$ is trivial if $G$ is replaced by a suitable finite-index subgroup
(cf.~\cite[Theorem 14.10]{Ue}).
Hence $G$ acts faithfully on a general fibre $S$ of $\psi$.
Under the identification $G \cong G \, | \, S$,
we have $N(G \, | \, S) = N(G) \, | \, S$ (cf.~\cite[2.1(11) Remark]{JDG}).
Since $G \ne N(G)$, our $G \, | \, S$ is not of null entropy.
Hence $\dim S \ge 2$. Also $\dim S = \dim X - \dim Y \le \dim X - 1 = 2$.
Thus $\dim S = 2$.

In the notation above, $(R_1 \, | \, S)/N(R_1 \, | \, S) \cong \Z$.
Hence the restrictions of $R_1$ and $R$ on $\NS_{\C}(S)$ are virtually solvable
(cf.~Theorem \ref{ThB} or \ref{ThC}).
Lemma \ref{stab} is applicable to $R \, | \, S \lhd G \, | \, S$
since the conditions in Lemma \ref{stab} (2) (for surfaces)
follow from the condition in Lemma \ref{stab} (1). Indeed, by the cone theorem of
Lie-Kolchin type \cite[Theorem 1.6]{Z-Tits},
$R_1$ (replaced by a finite-index subgroup) has a common nonzero nef eigenvector $L_1$;
thus the class $h^*L_1$ is parallel to $L_1$;
after switching $h$ with $h^{-1}$ if necessary,
\cite[Lemma 2.12]{antiK} implies that $h^*L_1 = d_1(h) L_1$, and also the
second condition in Lemma \ref{stab} (2).
So, by Lemma \ref{stab}, replacing $G$ by its subgroup of index $\le 2$,
$L_h$ gives rise to a character $\chi : G \, | \, S \to (\R_{> 0}, \times)$ and
that (for surfaces) $\Ker(\chi) = N(G \, | \, S)$. So
the null set $N(G \, | \, S)$ is a subgroup and
$G/N(G) = (G \, | \, S)/N(G \, | \, S) \cong \Imm \chi$, an abelian group.
Hence $G \, | \, L_{\C}$ is virtually solvable (cf.~Theorem \ref{ThB} or \ref{ThC}),
contradicting the assumption.
\end{proof}

In notation of \ref{setup2.1}, there exist two nonzero nef divisors $L_{h}$, $L_{h^{-1}}$
(which will be fixed) such that
$(h^{\pm1})^* L_{h^{\pm1}} = d_1(h^{\pm1}) L_{h^{\pm1}}$ with $d_1(h^{\pm1}) > 1$.

\begin{claim}\label{EV}
Suppose
there are a nef $\R$-divisor $M$, a real number $\lambda \ne 1$ and
an integer $s \ne 0$
such that $(h^s)^*M = \lambda M$ and $M$ is not parallel to $L_{h^{\pm1}}$.
Then Theorem $\ref{Ssimple} (2)$ holds.
\end{claim}

\begin{proof}
Note that $(h^s)^*L_{h} = d_1(h)^s L_{h}$,
$(h^s)^*L_{h^{-1}} = d_1(h^{-1})^{-s} L_{h^{-1}}$,
and $(h^s)^* M = \lambda^s M$.
Rewriting $h^s$ as $h$, we may assume $s = 1$.
We have $M . c_2(X) = h^*(M . c_2(X)) = h^*M . h^*c_2(X) = \lambda M . c_2(X)$.
Hence $M . c_2(X) = 0$ for $\lambda \ne 1$.  Similarly, $L_{h^{\pm1}} . c_2(X) = 0$.
By the assumption, $M . L_{h^{\pm1}} \ne 0$ (cf.~\cite[Corollary 3.2]{DS}).
Thus, since $M, L_h, L_{h^{-1}}$ are nef eigenvectors of $h^*$ corresponding to
eigenvalues $\lambda$, $d_1(h)$, $1/d_1(h^{-1})$ and since
$d_1(h) \ne 1/d_1(h^{-1})$,
\cite[Lemma 4.4]{DS} implies that the product of these three nef divisors is nonzero
and hence the sum of these three is a nef and big divisor, perpendicular to $c_2(X)$.
So Theorem \ref{Ssimple}(2) holds true, by Theorem \ref{Torus}(1) and Lemma \ref{K0}.
\end{proof}

We return to the proof of Theorem \ref{Ssimple}.
As proved above, the closed cone
$\Nef(X) \cap c_2(X)^{\perp} = \{M \in \Nef(X) \, | \, M . c_2(X) = 0\}$
contains $L_{h^{\pm1}}$. Since $R_1 | \NS_{\C}(X)$ is solvable,
the cone theorem of Lie-Kolchin type (cf.~e.g.~\cite[Theorem 2.6]{Z-Tits})
implies that the above closed cone contains a nonzero common nef divisor $L_1$
(with $L_1 . c_2(X) = 0$)
of $R_1$, after $R_1$ is replaced by a finite-index subgroup.
Write $g^*L_1 = \chi(g) L_1$ and consider the homomorphism
$$
\varphi: R_1 \, \to \, (\R, +), \,\,\,\,
g \, \mapsto \, \log \chi(g) .
$$

Clearly, $N(R_1) \le \Ker(\varphi)$ (cf.~\ref{setup2.1}).
If $g \in \Ker(\varphi) \setminus N(R_1)$, then
the product of the three nef eigenvectors $L_1$, $L_{g^{\pm1}}$
(corresponding to different eigenvalues $1$, $d_1(g)  \ne 1/d_1(g^{-1})$ of $g^*$)
is nonzero by \cite[Lemma 4.4]{DS} and hence the sum of these three
vectors is a nef and big $\R$-divisor class perpendicular to $c_2(X)$.
Thus, by Lemma \ref{K0}, we can apply Theorem \ref{Torus}(1) to conclude
Theorem \ref{Ssimple}(2).

Therefore, we may assume that $\Ker(\varphi) = N(R_1)$.
In particular, $\chi(h) \ne 1$, where $h^*L_1 = \chi(h) L_1$.
Possibly switching $h$ with $h^{-1}$, we may assume that
$\chi(h) > 1$. By Claim \ref{EV}, we may assume that $L_1 = L_h$
which is a common eigenvector of $R_1$.
Thus the condition (1) of Lemma \ref{stab} is satisfied
while the condition (2) can be assumed in view of Claim \ref{EV}.
So, by Lemma \ref{stab},
replacing $G$ by its subgroup of index $\le 2$, we may assume that
$G$ fixes $L_h$ up to scalars.
Write $g^*L_h = \chi'(g) L_h$.
Let
$$
\psi : G \, \to \, (\R, +), \,\,\,\,
g \, \mapsto \, \log \chi'(g)
$$
so that $G/\Ker(\psi)$ is mapped to an abelian subgroup of $(\R, +)$.
If $\Ker(\psi) \ne N(G)$, then as in the case of $\Ker(\varphi) \ne N(G)$ above,
we take $g \in \Ker(\psi) \setminus N(G)$, so $c_2(X)$ is perpendicular to the
nef and big divisor $L_1 + L_g + L_{g^{-1}}$ and hence
Theorem \ref{Ssimple}(2) occurs.

Thus we may assume that $N(G) = \Ker(\psi)$
which is hence a subgroup of $G$.
Since $G/N(G) \cong \Imm \psi$ is abelian,
$G \, | \, \NS_{\C}(X)$ is virtually solvable by Theorem \ref{ThB} or \ref{ThC},
contradicting the assumption.
{\it The proof of Theorem $\ref{Ssimple}$ is completed.}

\begin{setup}
{\bf Proof of Theorem \ref{Cor3}}
\end{setup}

We may assume that Theorem \ref{Ssimple}(2) occurs and use the notation there.
Let $\widetilde{G}$ be the lifting to $T$ of $G \, | \, X'$ with $\widetilde{G}/F = G \, | \, X'$.
As sets (and set of left cosets), we have equalities
$N(\widetilde{G})/F = N(G \, | \, X') = N(G) \, | \, X'$; so $N(G) \le G$ if and only if
$N(\widetilde{G}) \le \widetilde{G}$, and if this is the case
$\widetilde{G}/N(\widetilde{G}) \cong G/N(G)$
(cf.~\cite[Lemma 2.6]{JDG}). Thus, by Theorem \ref{ThC}(3),
as on $X$, neither $G \, | \, \NS_{\C}(X')$ nor $\widetilde{G} \, | \, \NS_{\C}(T)$
is virtually solvable. By the same reasoning, the
lifting to $T$ of $R(G) \, | \, X'$
has virtually solvable action on $ \NS_{\C}(T)$,
is normal in $\widetilde{G}$ and is not of null entropy.
$R(\widetilde{G})$ contains this lifting up to finite index, so it is not of null entropy.
Hence we may assume that $X = T$,
a complex torus.

Let
$\hat{G} \le \GL(H^0(X, \Omega_X^1)^{\vee}) = \GL_3(\C)$
be the Zariski-closure of the action
$G \, | \, H^0(X, \Omega_X^1)^{\vee}$. Since every $g \in G$ acts on $H^1(X, \Z)$ invertibly,
its matrix representation
has determinant $\pm 1$; note also $H^1(X, \C) = H^0(X, \Omega_X^1)$
$\oplus H^0(X, \Omega_X^1)^{\vee}$;
hence we may assume that $\hat{G}$ is contained in $\SL_3(\C)$
and connected, after $G$ is replaced by a finite-index subgroup.

Since $H^*(X, \C) := \oplus_{i \ge 0} H^i(X, \C)$
is generated by wedge products of $H^0(X, \Omega_X^1)$ and its conjugate,
the null set $N(G)$ is equal to
$\{g \in G \, ; \, g \, | \, H^0(X, \Omega_X^1)$ is of null entropy$\}$.
Let
$$R :=G \cap R(\hat{G})  \lhd G, \,\,\,\,
U := G \cap U(\hat{G})  \lhd G. $$
Then $R(\hat{G}) \, | \,  H^*(X, \C)$ and hence $R \, | \, H^*(X, \C)$ are solvable.
By Theorem \ref{ThC},
$R$ has a finite-index subgroup $R_1$ such that
$$\Z^{\oplus r} \cong R_1/N(R_1) \le R/N(R) .$$
Also $|N(R) : U| < \infty$ (cf.~Theorem \ref{ThB}). Thus $R/U$ contains a copy of
$\Z^{\oplus r}$ as a subgroup
of finite index (cf.~Lemma \ref{modfin}).
Consider the natural embedding
$$G/U \to J := \hat{G}/U(\hat{G})$$
into the reductive group
$J$ of real rank $\le \rk_{\R} \SL_3(\C) = 2$.

If $r \ge 2$, then $\rk_{\R} J = 2$; the Zariski-closure of $R/U \subset J$
contains a copy of $\Z^{\oplus 2}$ and hence a maximal torus of $J$, and is normal in $J$,
so this closure equals $J$. Thus $J$ (like $R/U$) and hence the actions of
$G/U$ and
$G$ on $H^*(X, \C)$ are all solvable, contradicting the  assumption.

Consider the case $r = 0$, i.e., $R \subseteq N(G)$.
This contradicts the extra assumption that $R(G)$ is not of null entropy and the assertion(*):
$R(G)$ equals $R$ up to finite index. Indeed, as in the proof of Lemma \ref{G|NS},
$G \,  | \, \NS_{\C}(X) = G/K$ with $|K : G \cap \Aut_0(X)| < \infty$,
while $G \,  | \, H^0(X, \Omega_X^1)^{\vee} = G/(G \cap \Aut_0(X))$.
Hence $G \,  | \, \NS_{\C}(X) $ equals $G \,  | \, H^0(X, \Omega_X^1)^{\vee}$ modulo a finite group.
Now the assertion(*) follows from the definitions of $R(G)$ and $R$.

Finally, assume $r = 1$, i.e., $R_1/N(R_1) \cong \Z$.
For any $g_1 \in G$,
the group $G_1 := \langle g_1, R \rangle$ (replaced by its finite-index subgroup)
has $G_1 \, | \, H^*(X, \C)$ solvable and
$G_1/N(G_1) \cong \Z^{\oplus s}$ (cf.~Theorem \ref{ThC}).
We claim that $s \ge 2$ for some $g_1$.
If the claim is false, then for any $g_1 \in G$, we have $s \le 1$ and
hence $g_1^a = h^b n$ for some $a \ge 1$,  $n \in N(G)$,
where $\langle \bar{h} \rangle =$ $R_1/N(R_1) \le G_1/N(G_1)$.
Thus $g_1$ (mod $R$) has a positive power acting as a unipotent
element on $H^0(X, \Omega_X^1)^{\vee}$ because the same is true for $n \in N(G)$.
So the subgroup $G/R$ of an arithmetic subgroup of
$\hat{G}/R(\hat{G})$ (defined over $\Q$; cf.~\cite[ChI; 0.11]{Ma})
has a unipotent group $U(G/R)$ as its subgroup of finite index, by Burnside's theorem
as in \cite[Proposition 2.2]{Og07} or Theorem \ref{ThB}.
Thus its Zariski-closure $\hat{G}/R(\hat{G})$ is both unipotent and semi-simple
and hence trivial. So $G  \, | \, H^*(X, \C)$ is solvable, contradicting the assumption.

Thus the claim is true and hence some
$G_1 := \langle g_1, R \rangle$ has $G_1/N(G_1) \cong \Z^{\oplus s}$
for some $s \ge 2$. So by \cite[Paragraph before \S 2.8]{CY3}, $U(G_1)$ and hence
$N(G_1)$ and $N(R )$ act as finite groups on $H^0(X, \Omega_X^1)$ and also on
$H^*(X, \C)$ (cf.~Proof of Theorem \ref{ThB}).
Thus, when restricted on $H^0(X, \Omega_X^1)^{\vee}$, our $R$
(containing a finite-index subgroup $R_1$ with $R_1/N(R_1) \cong \Z$)
is virtually infinite cyclic and normalized by
$G$, so it is contained in the centre of $G \, | \, H^0(X, \Omega_X^1)^{\vee}$
and of $\hat{G}$,
by replacing $G$ by a finite-index subgroup and
considering the conjugate action on the derived series of $R$.

Now we follow referee's suggestion.
Take an element $h \in R \setminus N(R )$.
If $h  \, | \, H^0(X, \Omega_X^1)^{\vee}$  $\in \SL_3(\C)$ has
three distinct eigenvectors, then $h$ and all elements of
$G$ are simultaneously diagonalizable and hence $G  \, | \, H^*(X, \C)$ is abelian,
contradicting the assumption.

Therefore, relative to a suitable basis $B$ of $H^0(X, \Omega_X^1)^{\vee}$,
our $h  \, | \, H^0(X, \Omega_X^1)^{\vee}$ is in one of the
Jordan canonical forms
$$\block \diag[\alpha^{-2}, J_2(\alpha)], \,\,\,\, \diag[\alpha^{-2}, \alpha, \alpha]$$
and the matrix representation
$g  \, | \, H^0(X, \Omega_X^1)^{\vee} = (a_{ij})$ of every $g \in G$
is especially upper triangular.
Consider the projection
$$\tau : G \to \C^*, \,\,\,\, g \mapsto a_{11} .$$
If $\Ker \tau \subseteq N(G)$,
then the actions of $\Ker \tau$ and hence of $(\Ker \tau) \, R$ on $H^0(X, \Omega_X^1)^{\vee}$
are virtually solvable (cf.~Theorem \ref{ThB}), so is that of $G$, because
$G/((\Ker \tau) \, R)$ is a quotient of the abelian group $\Imm \tau$.
This contradicts the assumption.

Thus, we can take $g_1 \in \Ker \tau \setminus N(G)$.
Then $g_1 \,  | \, H^0(X, \Omega_X^1)^{\vee}$ has 3 eigenvalues $1, \lambda^{\pm1}$
(with $|\lambda|  \ne 1$);
it has a unique (up to scalar) eigenvector $w \in H_1(X, \Z)$ ($= $ the lattice $\Lambda$ of the torus
$X = \C^3/\Lambda$) corresponding to the eigenvalue $1 \in \Q$ and is proportional to the column vector
$(1, 0, 0)^t$ (in basis $B$). Now $h$ or $h^{-1}$ takes $w$ to $\alpha^{-2} w$ with $|\alpha^{-2}| < 1$.
This contradicts the fact that $h(\Lambda) = \Lambda$ which is discrete in $\C^3 = \R^6$.
{\it Theorem $\ref{Cor3}$ is proved.}

\end{document}